\documentclass[11pt]{amsart}
\usepackage{amsmath, amssymb, amsthm, bm, mathtools, enumitem, caption}
\usepackage{hyperref}

\usepackage[noabbrev,capitalize]{cleveref}
\usepackage{adjustbox}
\crefname{equation}{}{}
\usepackage{ytableau}
\usepackage{graphicx, tikz}
\usepackage[textheight=8.75in, textwidth=6.85in]{geometry}

\usepackage{microtype}

\newtheorem{theorem}{Theorem}[section]
\newtheorem{lemma}[theorem]{Lemma}

\newtheorem{proposition}[theorem]{Proposition}
\newtheorem{conjecture}[theorem]{Conjecture}
\newtheorem*{conjecture*}{Conjecture}

\theoremstyle{definition}

\newtheorem{question}{Question}

\theoremstyle{remark}
\newtheorem*{remark}{Remark}

\newtheorem*{example}{Example}
\newtheorem*{examples}{Examples}

\numberwithin{equation}{section}

\DeclareMathOperator{\Tr}{Tr}

\newcommand{\C}{\mathbb C}
\newcommand{\SL}{\mathrm{SL}}

\newcommand{\Z}{\mathbb Z}

\title[Quasimodular forms and symmetric polynomials]{quasimodular forms arising from Jacobi's theta function and special symmetric polynomials}

\def\be{\begin{equation}}
\def\ee{\end{equation}}
\def\bea{\begin{eqnarray}}
\def\eea{\end{eqnarray}}

\author{Tewodros Amdeberhan$^{\star}$, Leonid G. Fel \and Ken Ono}
\address{Dept. of Mathematics, Tulane University, New Orleans, LA 70118, USA}
\email{tamdeber@tulane.edu}
\address{Dept. of Civil Engineering, Technion -- Israel Institute of Technology, Haifa 32000, Israel}
\email{lfel@technion.ac.il}
\address{Dept. of Mathematics, University of Virginia, Charlottesville, VA 22904, USA}
\email{ko5wk@virginia.edu}

\begin{document}

\maketitle

\begin{abstract} Ramanujan derived a sequence of even weight $2n$ quasimodular forms $U_{2n}(q)$ from derivatives of Jacobi's weight $3/2$ theta function. Using the generating function for this sequence, one can construct sequences of quasimodular forms of all nonnegative integer weights with minimal input: a weight 1 modular form and a power series $F(X)$. Using the weight 1 form $\theta(q)^2$ and $F(X)=\exp(X/2)$, we obtain a sequence $\{Y_n(q)\}$ of weight $n$ quasimodular forms on $\Gamma_0(4)$ whose symmetric function avatars $\widetilde{Y}_n(\pmb{x}^k)$ are the symmetric polynomials $T_n(\pmb{x}^k)$ that arise naturally in the study of syzygies of numerical semigroups. 
With this information,  we settle two conjectures  about the $T_n(\pmb{x}^k).$ Finally, we note that these polynomials are systematically given in terms of the Borel-Hirzebruch $\widehat{A}$-genus for spin manifolds, where one identifies power sum symmetric functions $p_i$ with Pontryagin classes. 
\end{abstract}

\section{Introduction and Statement of Results}
Using the classical weight 3/2 Jacobi theta function \cite[Thm. 1.60]{CBMS}
$$
\prod_{n=1}^{\infty} (1-q^n)^3 =
\sum_{k=0}^{\infty} (-1)^k (2k+1)q^{\frac{k(k+1)}{2}}=1-3q+5q^3-7q^6+\dots,
$$
Ramanujan (in his ``lost notebook" \cite[p. 369]{Rama})  defined the sequence of $q$-series
\begin{equation} \label{U}
U_{2n}(q)=\frac{1^{2n+1}-3^{2n+1}q+5^{2n+1}q^3-7^{2n+1}q^6+\cdots}{1-3q+5q^3-7q^6+\cdots}
=\frac{\sum_{k\geq0}(-1)^k(2k+1)^{2n+1} q^{\frac{k(k+1)}{2}}}{\sum_{k\geq0}(-1)^k (2k+1)q^{\frac{k(k+1)}2}}.
 \end{equation}
 He observed that
\begin{displaymath}
\begin{split}
&U_0=1,\ \ U_2=E_2,\ \ U_4=\frac13(5E_2^2-2E_4),\ \ {\text {\rm and}}\ \  U_6=\frac19(35E_2^3-42E_2E_4+16E_6),
\end{split}
\end{displaymath}
where $E_2, E_4,$ and $E_6$ are the classical Eisenstein series
\begin{displaymath}
\begin{split}
E_2(q):=1-24\sum_{n=1}^{\infty} \sum_{d\mid n}dq^n,\ \ 
E_4(q):=1+240\sum_{n=1}^{\infty}\sum_{d\mid n}d^3q^n, \ \  {\text {\rm and}}\ \ 
E_6(q):=1-504\sum_{n=1}^{\infty} \sum_{d\mid n}d^5q^n.
\end{split}
\end{displaymath}
Ramanujan also conjectured that each $U_{2n}(q)$ is a weight $2n$ quasimodular form on $\SL_2(\Z)$, a claim confirmed by Berndt et al. \cite{Berndt1, Berndt2} in the early 2000s. 
Answering a question of Andrews and Berndt (see p.~364 of \cite{AndrewsBerndt}), 
Singh and two of the authors recently found explicit formulas \cite{AOS} for each $U_{2n}(q)$ as {\it traces of partition Eisenstein series} (see (\ref{U_Trace})).

The key to these formulas is the discovery of the generating function (see Th. 3.4 of \cite{AOS})\footnote{In \cite[Thm. 3.4]{AOS}  we have replaced $\sin$ with hyperbolic $\sinh$ to eliminate the $(-1)^t$ factor.}
\begin{equation}\label{UGenFunction}
\Omega(X):=\sum_{n=0}^{\infty}  U_{2n}(q)\cdot \frac{X^{2n}}{(2n+1)!}=\frac{\sinh X}X \cdot \prod_{j\geq1}\left[1-\frac{4(\sinh^2X)q^j}{(1-q^j)^2}\right].
\end{equation}
Since Ramanujan's forms all have even weight, this generating function becomes a device for producing infinite sequences of quasimodular forms of all nonnegative integer weights, where the coefficient of $X^n$ has weight $n$. This method requires minimal additional input: a choice of a weight 1 modular form, a nonzero complex number $\alpha$, and a formal power series $$F(X)=\sum_{m=0}^{\infty}a(m)X^m.$$ 

We illustrate this method for the congruence subgroup $\Gamma_0(4),$ the elements of  the ring (for example, see \cite{Zagier})
 $\C[\theta, E_2, E_4, E_6]=\C[\theta, E_2, E_4, E_6,\dots]$, where $\theta(q)$ is the weight 1/2 theta function
\begin{equation}\label{theta}
 \theta(q) = \prod_{k\geq1}(1-q^{2k})(1+q^{2k-1})^2=1+2q+2q^4+2q^9+2q^{16}+\cdots .
 \end{equation}
Using the weight 1 modular form
$\theta(q)^2$, we obtain an infinite sequence of quasimodular forms, $Y_n(F,\alpha;q)$, one for every nonnegative integer weight $n$. Indeed, the  coefficient of $X^n$ of 
\begin{equation}\label{QuasiModFormula}
\sum_{n=0}^{\infty} Y_n(F,\alpha;q)\cdot \frac{X^n}{n!}:=F(\theta(q)^2\cdot X)\cdot \Omega(\alpha X)
\end{equation}
is a weight $n$ quasimodular form on $\Gamma_0(4).$
By letting $F(X)=e^{X/2}$ and $\alpha=1/2$, we obtain
\begin{equation}
\sum_{n\geq0} Y_n(q)\cdot \frac{X^n}{n!}:=
 \exp\left(\frac{\theta(q)^2 X}{2}\right)\cdot \Omega\left(\frac{X}{2}\right),
\end{equation}
where for convenience, we let $Y_n(q):=Y_n\left (e^{X/2}, \frac{1}{2};q\right).$

\begin{theorem} \label{Omega_coeff} For each non-negative integer $n$, the following are true.

\noindent
(1) We have that ${Y}_n(q)=\frac{1}{n+1}+O(q)$ is a weight $n$ quasimodular form on $\Gamma_0(4)$.

\noindent
(2) We have that
$$  {Y}_n(q) = \frac1{2^n(n+1)} \sum_{k=0}^{\lfloor\frac{n}2\rfloor} \binom{n+1}{2k+1} \theta(q)^{2n-4k} U_{2k}(q).$$

\end{theorem}

\begin{examples} The first few $q$-series take the form $Y_0(q)=1,$ and
\begin{align*}
&{Y}_1(q) = \frac12\theta^2=\frac{1}{2}+2q+2q^2+2q^4+4q^5+2q^8+2q^9+\dots,\\
&{Y}_2(q) = \frac{3\theta^4+E_2}{12}=\frac{1}{3}-8q^4-24q^8-32q^{12}-56q^{16}-48q^{20}-\dots, \\
&{Y}_3(q)= \frac{\theta^6+\theta^2E_2}8 =\frac{1}{4}-q-13q^2-40q^3-73q^4-122q^5-\dots.\\
\end{align*}
\end{examples}

Each $Y_n(q)$ has a canonical symmetric function avatar, which we will relate to a symmetric polynomial arising from syzygies of numerical semigroups. To this end, we make use of recent research by two of the authors that linked \cite[Th. 1.5]{AGO} the sequence $\{U_{2n}(q)\}$ to the $\widehat{A}$-genus of spin manifolds, a special generating function of polynomials, discovered by Borel and Hirzebruch. This identification is not immediate because the $\widehat{A}$-genus is expressed in terms of polynomials in Pontryagin classes \cite{Hitchin} and not as $q$-series. The connection involves special symmetric function avatars of the sequence $\{U_{2n}(q)\}$, a result that we will modify and then apply to the $\{Y_n(q)\}$ using Theorem~\ref{Omega_coeff} (2).

These identifications rely on partition Eisenstein series.
For a partition $\lambda=(1^{m_1},\dots,n^{m_n})\vdash n$, where $m_j$ is the multiplicity of $j$, the {\it partition Eisenstein series} is given by (see \cite[(1.5)]{AOS} or \cite[(1.2]{AGOS})
\begin{equation}\label{PartitionEisenstein}
\lambda=(1^{m_1}, 2^{m_2},\dots, n^{m_n}) \vdash n \ \ \ \ \ \longmapsto \ \ \ \ \  E_{\lambda}:= E_2^{m_1}E_4^{m_2}\cdots E_{2n}^{m_n},
\end{equation} 
where the classical Eisenstein series (see Chapter 1 of \cite{CBMS}) are defined by
\begin{equation}\label{Eisenstein}
E_{2k}(q):=1-\frac{4k}{B_{2k}}\sum_{n=1}^{\infty}\sigma_{2k-1}(n)q^n,
\end{equation}
and $B_{2k}$ is the $2k$th Bernoulli number and $\sigma_v(n):=\sum_{d\vert n}d^v$. 
If $\phi$ is a function on integer partitions and $n$ is positive integer, then the {\it $n$th trace of partition Eisenstein series} \cite[(1.6)]{AOS} is defined by
$$\text{Tr}_n(\phi;q):=\sum_{\lambda\vdash n} \phi(\lambda)E_{\lambda}(q).$$
Singh and two of the authors \cite[Th. 1.3 (1)]{AOS} proved that
each $U_{2n}(q)$ is the trace
\begin{equation}\label{U_Trace}
U_{2n}(q)=\text{Tr}_n(\phi_U;q),
\end{equation}
where we let
$$
\phi_U(\lambda):= (2n+1)!\, 4^n
\cdot \prod_{j=1}^n\frac1{m_j!}\left(\frac{ B_{2j}}{(2j)(2j)!}\right)^{m_j}.
$$
Hence, Theorem~\ref{Omega_coeff} (2) can be reformulated in terms of traces of partition Eisenstein series.
\begin{theorem}\label{Omega_coeff2}
If $n$ is a nonnegative integer, then
$$  {Y}_n(q) = \frac1{2^n(n+1)} \sum_{k=0}^{\lfloor\frac{n}2\rfloor} \binom{n+1}{2k+1} \theta(q)^{2n-4k} \Tr_k(\phi_U;q).$$
\end{theorem}

\begin{remark} We could have declared $\theta(q)^2$ as an Eisenstein series allowing us to reformulate the sums in Theorem~\ref{Omega_coeff2} as partition Eisenstein traces over $\{\theta(q)^2, E_2, E_4, E_6,\dots\}$, along the lines of recent work by Bringmann, Pandey, and van Ittersum \cite{BPvI}. However, this would have required unnecessary new notation.
\end{remark}

We use these identities to obtain avatars of the $U_{2n}(q).$ The correspondence between quasimodular forms and symmetric functions is made by substituting\footnote{This identification is different from \cite{AGO}, where $p_k\longleftrightarrow E_{2k}$.} each
$$E_{2k}(q) \longleftrightarrow p_{2k}(\pmb{x}),$$
 in (\ref{PartitionEisenstein}) and (\ref{U_Trace}.  Here $p_{2k}(\pmb{x})=\sum_{i\geq 1} x_i^{2k}$ is the $2k$th \emph{even} power sum symmetric function, where we let $\pmb{x}=(x_1, x_2,\dots).$ In particular, we have
\begin{equation}\label{PsiProduct}
\Psi(E_{\lambda})=\prod_{i=1}^n \Psi(E_{2i})^{m_i}=p_{\lambda}(\pmb{x}) \qquad \text{where} \qquad
p_{\lambda}(\pmb{x}):=p_2(\pmb{x})^{m_1} p_4(\pmb{x})^{m_2}\cdots p_{2n}(\pmb{x})^{m_n},
\end{equation}
which we extend linearly to define the symmetric function representation of $\Psi(U_{2n}(q))$ using (\ref{U_Trace}).

In view of Theorem~\ref{Omega_coeff2}, to produce symmetric polynomial avatars of each $Y_n(q),$ we must also define $\Psi(\theta(q)^2)$. Although there are no weight 1 modular forms on $\SL_2(\Z)$, one can ``think of''  $\theta(q)^2$ as a proxy for $E_1(q)$ for several reasons. First, it has weight 1. Additionally, by setting $k=1/2$ and replacing 
$\sigma_{2k-1}(n)$ with $\sigma_* (n):=d_1(n)-d_3(n)$, where $d_j(n)$ is the number of divisors of $n$ congruent to $j$ modulo 4, in (\ref{Eisenstein}), we find
$$
\theta(q)^2=1-\frac{2}{B_1}\sum_{n=1}^{\infty} \sigma_*(n)q^n=1+4\sum_{n=1}^{\infty}\sigma_*(n)q^n.
$$
The only adjustment needed is a minor modification to the divisor function.
Therefore, it is consistent to let $\Psi: \theta^2(q)\mapsto p_1(\pmb{x})=x_1+x_2+\dots$. 
Hence, thanks to Theorem~\ref{Omega_coeff2}, we define the symmetric polynomial counterpart to
$Y_n(q)$ by
\begin{equation} \label{Yn_tilde}
\widetilde{Y}_n(\pmb{x})
:= \frac1{(n+1)2^n} \sum_{k=0}^{\lfloor\frac{n}2\rfloor} \binom{n+1}{2k+1} p_1(\pmb{x})^{n-2k}\cdot \Psi(U_{2k}(q)).
\end{equation}
The first few examples are as follows:
\begin{align*}
\widetilde{{Y}}_0(\pmb{x})&=1, \ \ \ 
\widetilde{Y}_1(\pmb{x}) = \frac12p_1(\pmb{x}), \ \ \ 
\widetilde{{Y}}_2(\pmb{x}) = \frac{3p_1(\pmb{x})^2+p_2(\pmb{x})}{12}, \\
\widetilde{{Y}}_3(\pmb{x}) &= \frac{p_1(\pmb{x})^3+p_1(\pmb{x})p_2(\pmb{x})}8, \qquad
\widetilde{{Y}}_4(\pmb{x})= \frac{15p_1(\pmb{x})^4+30p_1(\pmb{x})^2p_2(\pmb{x})+5p_2(\pmb{x})^2-2p_4(\pmb{x})}{240}.
\end{align*}

In Section~\ref{TrSection}, we define symmetric functions $T_n(\pmb{x}^k)$ (defined in \eqref{a10}  using polynomials $P_m(\pmb{x}^k)$ given by \eqref{a9}), where $\pmb{x}^k=(x_1, x_2,\dots, x_k)$. These functions arise from a study of the alternating power sums of syzygies degrees of numerical semigroups.
In \cite{Fel}, the second author explicitly computed many of these functions, where the first include the following.
\begin{align*}
T_0(\pmb{x}^k)&=1, \qquad
T_1(\pmb{x}^k) = \frac12p_1(\pmb{x}^k), \qquad
T_2(\pmb{x}^k) = \frac{3p_1(\pmb{x}^k)^2+p_2(\pmb{x}^k)}{12}, \\
T_3(\pmb{x}^k) &= \frac{p_1(\pmb{x}^k)^3+p_1(\pmb{x}^k)p_2(\pmb{x}^k)}8, \ \ \
T_4(\pmb{x}^k) = \frac{15p_1(\pmb{x}^k)^4+30p_1(\pmb{x}^k)^2p_2(\pmb{x}^k)+5p_2(\pmb{x}^k)^2-2p_4(\pmb{x}^k)}{240}.
\end{align*}
One immediately notices the similarity between these symmetric functions. It is not an accident. 

\begin{theorem}\label{T_and_Y} For every pair of positive integers $k$ and $n$,   we have  $\widetilde{{Y}}_n(\pmb{x}^k) = T_n(\pmb{x}^k)$.
\end{theorem}

\begin{remark} The $\widetilde{Y}_n(\pmb{x})$ are defined with infinitely many variables $\pmb{x}=(x_1, x_2,\dots).$ In Theorem~\ref{T_and_Y}, we replace $\pmb{x}$ by $\pmb{x}^k=(x_1,\dots, x_k),$ which is required in the context arising from numerical semigroups.
\end{remark}

Theorem~\ref{T_and_Y}, which stems from Ramanujan's $U_{2n}(q),$ unexpectedly offers deep insight into the theory of these symmetric polynomials. This knowledge allows us to solve two conjectures that the second author \cite{Fel}
formulated about $T_n(\pmb{x}^k)$. The first is that the $T_n(\pmb{x}^k)$ are modified versions of symmetric polynomials $f_n(\pmb{x}^k)$ that arise from (restricted) integer partitions (see\eqref{W1_f} and \eqref{black} for the definition and additional properties).

\begin{conjecture}{{\text {\rm [Conjecture 2.1 of \cite{Fel}]}}} If $n\geq 2$, then we have
$$
T_n(p_1,p_2,p_3,p_4,\dots,p_n)=f_n(p_1,-p_2,p_3,-p_4,\dots,p_n).
$$
\end{conjecture}
\begin{theorem}\label{Conjecture1} Conjecture 1.4 is true.
\end{theorem}

The second conjecture of \cite{Fel} claimed striking identities. 
To state the conjecture\footnote{We note that we offer a more precise formulation of the conjecture  \cite[p. 62]{Fel}.}, we require the numbers $A_{2j+1},$ known as the \emph{tangent/zig-zag numbers} (counting up/down permutations), which have the exponential generating function
$$\sec x+\tan x= \sum_{j\geq0} A_j \cdot \frac{x^j}{j!}.$$

\begin{conjecture}{{\text {\rm [Conjecture 3.1 of \cite{Fel}]}}}  For a positive integer $n$, we have
$$\frac{T_{2n+1}(\pmb{x}^k)}{T_1^{2n+1}(\pmb{x}^k)}
= \sum_{j=0}^n (-1)^j A_{2j+1}\binom{2n+1}{2j+1} \frac{T_{2n-2j}(\pmb{x}^k)}{T_1^{2n-2j}(\pmb{x}^k)}.$$
\end{conjecture}

\begin{theorem}\label{Conjecture2} Conjecture 1.6 is true.
\end{theorem}

These results illustrate the rich mathematics that arises from Ramanujan's $U_{2n}(q)$ by
letting $F(X)=\exp(X/2)$ and
$\alpha=1/2$ in (\ref{QuasiModFormula}).
Hence, it is natural to ask the following open-ended questions.

\begin{question} In view of Theorem~\ref{Omega_coeff2} and \cite[Th. 1.4]{AGO}, it is natural to ask whether the sequence $\{T_n(\pmb{x}^k)\}$ encodes algebraic information about the 
Pontryagin classes \cite{Hitchin} of spaces assembled from spin manifolds, perhaps through the framework of syzygies of the genera of numerical semigroups.
\end{question}

\begin{question}
Do other choices of weight 1 modular forms and power series $F(X)$ in (\ref{QuasiModFormula}) produce further important sequences of symmetric polynomials that are independent of the power sum symmetric functions $p_3, p_5, p_7,\dots$?
\end{question}

\begin{question}
Does (\ref{QuasiModFormula}) generalize to other families of even weight quasimodular forms (i.e., replacing $U_{2n}(q)$ and the generating function $\Omega(X)$)? Is there a general theory, perhaps arising from the theory of Jacobi forms? 
\end{question}

The proofs of these results make use of tools from umbral calculus, P\'olya's cycle index theorem and involutions on symmetric functions. In Section~\ref{NutsBolts}, we recall the nuts and bolts
that we require. Namely, in Subsection~\ref{TrSection}
we offer the required background on numerical semigroups and the symmetric polynomials $T_{r}(\pmb{x}^k)$ and their algebraic properties, and in Subsection~\ref{pr_section}, we recall the symmetric functions $f_r({\bf d}^k)$ that arise from restricted integer partitions. Finally, in Section~\ref{Proofs} we compile these results to prove Theorems~\ref{Omega_coeff}, \ref{Omega_coeff2},  \ref{T_and_Y},  \ref{Conjecture1}, and \ref{Conjecture2}.

\section*{Acknowledgements}
 The authors thank Toshiki Matsusaka, Badri Pandey, Ajit Singh for {\color{red}their} 
constructive comments on earlier versions of this paper. T.A. especially thanks Ira Gessel and Christophe Vignat for valuable discussions. The third author thanks the Thomas Jefferson Fund and the NSF
(DMS-2002265 and DMS-2055118).

\section{Nuts and Bolts}\label{NutsBolts}

In this section, we recall basic facts about two families of symmetric polynomials, one arising from numerical semigroups, and the other arising from restricted integer partitions.

\subsection{Symmetric functions arising from syzygies of numerical semigroups}\label{TrSection}

Let ${\mathbb N}$ denote the set of all non-negative integers.
A {\it numerical semigroup} $\langle d_1,\ldots,d_m\rangle$ is a subset $S_m\subset {\mathbb N}$ containing 0, closed under summation, with finite complement in ${\mathbb N}$. A set of generators $\{d_1,\ldots,d_m\}$ of a numerical semigroup $S_m$ 
satisfies $\gcd(d_1,\ldots,d_m)=1$ and is \emph{minimal} if none of its proper subsets generates the numerical semigroup $S_m$ \cite{fr86}. 
Its generating function is given by
\bea
H\left(S_m;z\right)=\sum_{s\;\in\;S_m}z^s,\qquad z<1,\qquad 0\in S_m,\label{y1}
\eea
and is known as {\em the Hilbert series} of $S_m$ and has a rational representation (Rep),
\bea
H\left(S_m;z\right)=\frac{Q\left(S_m;z\right)}{\prod_{i=1}^m\left(1-z^{d_i}
\right)},
\label{y2}
\eea
where 
\bea
Q\left(S_m;z\right)=1-\sum_{j=1}^{\beta_1}z^{C_{1,j}}+\sum_{j=1}^{\beta_2}      
z^{C_{2,j}}-\cdots\pm\sum_{j=1}^{\beta_{m-1}}z^{C_{m-1,j}},\qquad\sum_{k=0}^
{m-1}(-1)^k\beta_k=0,\label{y3}
\eea
while  $C_{k,j}\in{\mathbb N},\;\;1\le k\le m-1,\;\;1\le j\le \beta_k$,
and $C_{k,j}$ and $\beta_k$ stand for degrees of the $k$th syzygy and partial Betti's numbers, respectively \cite{st78}.
Here $\beta_0=1$.

\smallskip
We denote by ${\mathbb C}_k(S_m)$ the alternating power sum of syzygy degrees
\begin{equation}
{\mathbb C}_k(S_m):=\sum_{j=1}^{\beta_1}C_{1,j}^k-\sum_{j=1}^{\beta_2}C_{2,j}^k+\ldots-(-1)^{m-1}\sum_{j=1}^{\beta_{m-1}}C_{m-1,j}^k.  
\end{equation}
We have that \cite[Thm. 1]{fl17}
$$ {\mathbb C}_0(S_m)=1, \qquad
{\mathbb C}_r(S_m)=0,\quad 1\le r\le m-2,\qquad
{\mathbb C}_{m-1}(S_m)=(-1)^m(m-1)!\cdot \pi_m,
$$
where we let $\pi_m:=\prod_{i=1}^md_i$.
For $r\ge m,$ these sums 
 were calculated and furnished in \cite[(22)]{fl22}  as
\bea
{\mathbb C}_n(S_m)=\frac{(-1)^mn!}{(n-m)\:!}\;  \pi_m \cdot K_{n-m}(S_m),\qquad 
K_t(S_m)>0,\;\;t\ge 0,     \label{a5}
\eea
where $K_t(S_m)$ is a linear combination of the genera $G_0(S_m),\ldots,G_t(S_m)$ of the semigroup $S_m$, and $G_t(S_m):=\sum_{s\in\Delta_m}\!s^t$. 
Here,
$\Delta_m:={\mathbb N}\!\setminus S_m$ and $G_0(S_m)=\#\Delta_m$ denote a set
of gaps and a genus of $S_m$, respectively. 

A special kind of numerical semigroups $S_m,$ with the first Betti number $\beta_1=m-1,$ is called a symmetric {\em complete intersection} (CI), where the freedom to choose generators 
$\{d_1,\ldots,d_m\}$ of the semigroup $S_m^{CI}$ becomes restricted due to the relations
$$
\beta_k=\beta_{m-k-1},\quad\beta_1=m-1,\quad\beta_{m-1}=1,\quad C_{k,j}+C_{m-k-1,j}=C_{m-1,1},\quad 
\min\{d_1,\dots,d_m\}\geq m+1.
$$
The rational Rep of its Hilbert series reads
\bea
H\left(S_m^{CI};z\right)=\frac{\prod_{j=1}^{m-1}
\left(1-z^{e_j}\right)}
{\prod_{i=1}^m\left(1-z^{d_i}\right)}
\qquad \text{and} \qquad e_j\ge 2(m+1).     
\label{a13}
\eea
The tuple $\pmb{e}^{m-1}:=(e_1,\ldots,e_{m-1})$ presents the $m-1$ degrees of the first syzygy for symmetric CI semigroup $S_m^{CI}$. 

\smallskip
In contrast to $Q\left(S_m;z\right)$ of \eqref{y2}, \eqref{y3}), the numerator in (\ref{a13}) 
is the special symmetric polynomial $P_n({\bf e}^{m-1})$ 
introduced by the second author \cite{Fel}: 
\begin{align} \label{a9}
P_n(\pmb{x}^m)&:=\sum_{j=1}^mx_j^n-\sum_{1\le j<r}^m\left(x_j+x_r\right)^n+\sum_{1\le j<r<i}^m\left(x_j+x_r+x_i\right)^n-\ldots  -(-1)^m\left(\sum_{j=1}^mx_j\right)^n.
\end{align}
This polynomial allows \cite{Fel} for the construction of the polynomials $T_n(\pmb{x}^m)$ from
\bea
P_n(\pmb{x}^m)=\frac{(-1)^{m+1}n!}{(n-m)!}\;\chi_mT_{n-m}\left(\pmb{x}^m\right),\qquad \chi_m=\prod_{j=1}^mx_j,\quad 
T_0\left(\pmb{x}^m\right)=1,\label{a10}
\eea

These polynomials enjoy numerous properties, such as the following.

\begin{proposition}\cite[Lm. 1.3]{Fel}
We have the inequality
$T_n(x_1,\ldots,x_m)\ge 0$ whenever $x_1,\ldots,x_m\ge 0$.
\end{proposition}

\begin{remark} From \eqref{a10}, one might not expect that these symmetric functions are independent of $p_{2j-1}(\pmb{x}^k),$ where $j>1$, which is a consequence of Theorem~\ref{T_and_Y}.
\end{remark}

For completeness, we remind the reader on the importance of these polynomials $T_n(\pmb{x}^m)$.
They are conjectured (see Conjecture~\ref{conj2.1} below) to play a central role in the explicit description of alternating power sums of syzygy degrees for the genera of the semigroups $S_m$.

\begin{conjecture} \label{conj2.1} If we let $\sigma_j=\sum_{i=1}^md_i^j$ and 
$\delta_j=\frac{\sigma_j-1}{2^j}$, then we have
\bea
K_t(S_m)=\sum_{r=0}^t \binom{t}r T_{t-r}(\sigma)G_r(S_m)+\frac{2^{t+1}}{t+1}T_{t+1} \label{a26}
(\delta),
\eea
where $T_r(X):=T_r(X_1,\dots,X_r)$ stand for symmetric polynomials in the power sums
\bea
X_k(\pmb{x}^m):=\sum_{j=1}^mx_j^k \qquad \text{for an $m$-tuple $\pmb{x}^m:=(x_1,\dots,x_m)$}. 
\label{a27}
\eea
\end{conjecture}

\smallskip

\subsection{Symmetric functions arising from partitions}\label{pr_section}

\smallskip

Polynomials $T_n(\pmb{x}^k)$ are analogous to another sequence of polynomials $\{f_n\}$ that arise in the theory of integer partitions. Given ${\bf d}^k:=\{d_1,\ldots,d_k\}$, consider the restricted partition function $W(s;{\bf d}^k)$ (see \cite{Andrews}), which counts integer partitions of $s\ge 0$ into $k$ positive integers $(d_1,\ldots,d_k).$

\smallskip
It is well-known (see, for example \cite[Section 3.1]{fl17}) that $W(s;\pmb{d}^k)$ can be described in terms of finitely many quasipolynomials $W_q(s;\pmb{d}^k)$, each containing a single $q$-periodic function with $W_1(s;\pmb{d}^k)$ being a polynomial, given in the form
$$W(s;\pmb{d}^k)=\sum_{\substack{ q\vert d_i \\ 1\leq i\leq k}} W_q(s;\pmb{d}^k).$$
The $W_q(s;\pmb{d}^k)$'s are called the \emph{Sylvester waves}. In particular, $W_1(s;\pmb{d}^k)$ is referred to as the first Sylvester wave and has fascinating properties. For example, $W_1(s;\pmb{d}^k)$ is the polynomial part of 
$W(s;\pmb{d}^k)$ having an explicit formula \cite[(3.16) and (7.1)]{fl17} given by
\begin{align} \label{W1_f}
W_1(s;{\bf d}^k)=\frac1{(k-1)!\pi_k}\sum_{j=0}^{k-1} \binom{k-1}j f_j({\bf d}^k)s^{k-1-j},\qquad f_j({\bf d}^k)=\left(\sigma_1+\sum_{i=1}^k \mathcal{B}_id_i\right)^j,
\end{align}
where $\sigma_1:=\sum_{j=1}^kd_j$ and $\pi_k:=\prod_{i=1}^kd_i$. For our purpose in this paper, we will replace $\pmb{d}^k=(d_1,\dots,d_k)$ by $\pmb{x}^k=(x_1,\dots,x_k)$ and $\sigma_j$ by the power sum $p_j$ in the variables $(x_1,\dots,x_k)$.
 By treating the $\mathcal{B}_i$'s as independent random variables, the Bernoulli symbolism $(\mathcal{B}_ix_i)^n$ is understood as $\mathcal{B}_i^nx_i=B_i^nx_i^n$ where 
$$\frac{z}{\exp(z)-1}=\sum_{n\geq0} B_n\cdot \frac{z^n}{n!}
$$
generates the Bernoulli numbers.

\begin{examples}
The first few symmetric polynomials 
\begin{align} \label{black} f_n(\pmb{x}^k):=\left(p_1+\sum_{i=1}^k
\mathcal{B}_ix_i\right)^n
=f_n(p_1,\ldots,p_n), \end{align}
expanded in power sums (for example, see \cite[(7.2)]{fl17}), are:
\begin{align*}
f_0(\pmb{x}^k)&=1, \qquad
f_1(\pmb{x}^k) = \frac12p_1, \qquad
f_2(\pmb{x}^k) = \frac{3p_1^2-p_2}{12}, \qquad
f_3(\pmb{x}^k) = \frac{p_1^3 -p_1p_2}8, \\
f_4(\pmb{x}^k)& = \frac{15p_1^4 -30p_1^2p_2+5p_2^2 +2p_4}{240}, \qquad
f_5(\pmb{x}^k) = \frac{3p_1^5-10p_1^3p_2+5p_1p_2^2 +2p_1p_4}{96}, \\
f_6(\pmb{x}^k)& = \frac{63p_1^6 -315p_1^4p_2+315p_1^2p_2^2 +126p_1^2p_4-35p_2^3-42p_2p_4 -16p_6}{4032}.
\end{align*}
\end{examples}

\begin{remark} If we set all of the variables to $x$, then we obtain the  {\it convolution Stirling polynomials} studied by Knuth \cite{Knuth}
$$
f_1(x)=\frac{x}{2},\ \ f_2(x,x)=\frac{3x^2-x}{12}, \ \ f_3(x,x,x)=\frac{x^3-x^2}{8},\ \ f_4(x,\dots,x)=\frac{15x^4-30x^3+5x^2+2x}{240}\dots,
$$
which has generating function
$$
\left(\frac{z\exp(z)}{\exp(z)-1}\right)^x=\sum_{n=0}^{\infty}
f_n(x,\dots,x)\cdot \frac{z^n}{n!}.
$$
\end{remark}

\section{Proofs of Theorems~\ref{Omega_coeff}, \ref{Omega_coeff2}, \ref{T_and_Y}, \ref{Conjecture1}, and \ref{Conjecture2}}\label{Proofs}

Here we prove the results from the Introduction using the material in the previous section.

\subsection{Proof of Theorems~\ref{Omega_coeff} and \ref{Omega_coeff2}}
We consider (\ref{QuasiModFormula}) for a general power series
$$ F(X)=\sum_{m=0}^{\infty}a(m)X^m.
$$
Since we have
$$
\Omega(\alpha X)=\sum_{t=0}^{\infty} U_{2t}(q)\cdot \frac{\alpha^{2t} X^{2t}}{(2t+1)!},
$$
it follows that
$$
Y_n(F,\alpha;q)=\sum_{m+2t=n}
\frac{a(m)\cdot n! \alpha^{2t}}{(2t+1)!}\cdot \theta(q)^{2m}U_{2t}(q).
$$
Since $\theta(q)$ is a weight 1/2 modular form on $\Gamma_0(4)\subseteq \SL_2(\Z)$ (for example, see Chapter 1 of \cite{CBMS}) and $U_{2t}(q)$ is a weight $2t$ quasimodular form on $\SL_2(\Z)$, it follows that each $\theta(q)^{2m}U_{2t}(q)$ is a weight $m+2t=n$ quasimodular form on $\Gamma_0(4),$ which in turn implies that $Y_n(F,\alpha;q)$ is as well.

The proof of Theorem~\ref{Omega_coeff} (1-2) follows easily from the Taylor expansion of $F(X)=\exp(X/2)$ and direct algebraic manipulation.
To obtain Theorem~\ref{Omega_coeff2}, one simply substitutes (\ref{U_Trace}) (see also \cite[Thm. 1.3 (1)]{AOS}) into the formulas in Theorem~\ref{Omega_coeff} (2).

\subsection{Proof of Theorem~\ref{Conjecture1}}

We require a few elementary facts from probability theory in the context of umbral calculus (for background on umbral calculus, see \cite{RR}). 

\begin{lemma} \label{lemma1} If we let $\mathcal{B}^n=B_n$ (Bernoulli numbers), $\mathcal{C}_n:=\frac{(-1)^n}{n+1}$ and $\mathcal{D}_n:=\frac1{n+1}$, then for positive integers $n$ in umbral formalism (treating $\mathcal{B}, \mathcal{C}$ and $\mathcal{D}$ as random variables) the following are true:

\smallskip
\noindent
(1) $(\mathcal{B}+\mathcal{D})^n=0$.

\noindent
(2) $(\mathcal{B}+\mathcal{C})^n=(-1)^n$.

\noindent
(3) $1+\mathcal{C}=\mathcal{D}$.

\noindent
(4) $1+\mathcal{B}=-\mathcal{B}$. 

\noindent
(5) $(-\mathcal{B})^n=\mathcal{B}^n$ for $n\neq1$ with with $B_1 = -\frac12$.
\end{lemma}
\begin{proof} Since these claims are straightforward, we only prove two of them.
To prove (4), we recall the generating function $$\exp(\mathcal{B}z)
=\sum_{n\geq0}\mathcal{B}^n\cdot \frac{z^n}{n!}=\sum_{n\geq0} B_n\cdot \frac{z^n}{n!}=\frac{z}{\exp(z)-1}.$$
This implies that
 $$\exp((1+\mathcal{B})z)=\frac{z\exp(z)}{\exp(z)-1}=\frac{-z}{\exp(-z)-1}=\exp((-\mathcal{B})z).$$
Similarly, we have that
$$\exp(\mathcal{D}z)=\sum_{n\geq0}\mathcal{D}^n\frac{z^n}{n!}=\sum_{n\geq0}\frac1{n+1}\frac{z^n}{n!}=\frac{\exp(z)-1}z.
$$
Therefore, (1) follows from the identity
 $$\exp((\mathcal{B}+\mathcal{D})z)=\frac{z}{\exp(z)-1}\cdot \frac{\exp(z)-1}z=1=\exp(0z),$$
 which means that $\mathcal{B}$ and $\mathcal{D}$ annihilate each other. 
\end{proof}

Using this lemma, we provide an umbral expression for $T_n(\pmb{x}^k)$ in the spirit of \eqref{black}. 

\begin{lemma} \label{conjTA} If $\mathcal{C}_i=\frac{(-1)^i}{i+1}$ is a sequence and $ k, n\in\mathbb{N}$, then we have
\begin{align*} T_n(\pmb{x}^k)=\left(p_1+\sum_{i=1}^k\mathcal{C}_i x_i\right)^n. \end{align*}
\end{lemma}

\begin{proof}
Given a partition $\lambda=(\lambda_1,\lambda_2,\dots)$, let $m_{\lambda}$ denote the associated \emph{monomial symmetric function} \cite[Chapter 7]{RPS} 
by $m_{\lambda}$ defined by
$$m_{\lambda}:=\sum_{\alpha} x_1^{\alpha_1}x_2^{\alpha_2}\cdots$$
where the sum ranges over all \emph{distinct} permutations $\alpha=(\alpha_1,\alpha_2,\dots)$ of the entries of $\lambda$. For example, $m_{(1,1)}=\sum_{i<j} x_ix_j$ and $m_{(2,1,1)}=\sum x_i^2x_jx_k$ with $i, j, k$ distinct.

Given a sequence $\mathcal{U}_1,\mathcal{U}_2, \dots$, we write $\mathcal{U}_{\lambda}=\mathcal{U}_{\lambda_1}\mathcal{U}_{\lambda_2}\cdots$. 
From Lemma ~\ref{lemma1} (3), we obtain $$\left (p_1+\sum_i \mathcal{C}_ix_i\right)^n=\left (\sum_i\mathcal{D}_ix_i\right)^n.
$$
Then, the multinomial theorem and umbral calculus together gives
\begin{align} \label{mono}
\left(\sum_{i=1}^k\mathcal{D}_i x_i\right)^n
& = \sum_{\lambda\vdash n} \binom{n}{\lambda_1,\lambda_2,\dots} \mathcal{D}_{\lambda} m_{\lambda}.
\end{align} 
We apply \eqref{mono} with $\mathcal{D}$ replaced by $\pmb{1}=(1,1,\dots)$ and $n$ replaced by $n':=n+k$, while we rank partitions according to their \emph{lengths}, to reformulate \eqref{a9} to obtain
\begin{align*}
\mathcal{P}_{n'}(\pmb{x}^k)
& = \sum_{t=1}^k (-1)^{t-1}\sum_{1\leq j_1<\cdots<j_t\leq n'} (x_{j_1}+\cdots+x_{j_t})^{n'} \\
& = \sum_{t=1}^k (-1)^{t-1}\sum_{1\leq j_1<\cdots<j_t\leq n'}  \sum_{\ell=1}^t \sum_{\substack{1\leq s_1<\cdots<s_{\ell}\leq n' \\ s_1,\dots,s_{\ell}\in\{j_1,\dots,j_t\}}}
\sum_{\substack{\mu\vdash n' \\ \ell(\mu)=\ell}} \binom{n}{\mu_1,\dots,\mu_{\ell}} m_{\mu}(x_{s_1},\dots,x_{s_{\ell}}) \\
& = \sum_{\ell=1}^k \sum_{t=\ell}^k (-1)^{t-1} \sum_{1\leq j_1<\cdots<j_t\leq n'} \sum_{\substack{1\leq s_1<\cdots<s_{\ell}\leq n' \\ s_1,\dots,s_{\ell}\in\{j_1,\dots,j_t\}}}
\sum_{\substack{\mu\vdash n' \\ \ell(\mu)=\ell}} \binom{n'}{\mu_1,\dots,\mu_{\ell}} m_{\mu}(x_{s_1},\dots,x_{s_{\ell}}) \\
& = \sum_{\ell=1}^k \sum_{t=\ell}^k (-1)^{t-1} \sum_{1\leq j_1<\cdots<j_t\leq n'} 
\sum_{\substack{\mu\vdash n' \\ \ell(\mu)=\ell}} \binom{n'}{\mu_1,\dots,\mu_{\ell}} 
\left\{ \sum_{\substack{1\leq s_1<\cdots<s_{\ell}\leq n' \\ s_1,\dots,s_{\ell}\in\{j_1,\dots,j_t\}}} m_{\mu}(x_{s_1},\dots,x_{s_{\ell}})  \right\}.  
\end{align*}
For a given $t$, ranging in $\{1,\dots,k\}$, counting the subsets $\{j_1,\dots,j_t\}$ that contain $\{s_1,\dots,s_{\ell}\}$ leads to
\begin{align*}
\mathcal{P}_{n'}(\pmb{x}^k)
&= \sum_{\ell=1}^k
\sum_{\substack{\mu\vdash n' \\ \ell(\mu)=\ell}} \binom{n'}{\mu_1,\dots,\mu_{\ell}} 
\sum_{1\leq s_1<\cdots<s_{\ell}\leq n'} m_{\mu}(x_{s_1},\dots,x_{s_{\ell}})   
\sum_{t=\ell}^k (-1)^{t-1}\binom{k-\ell}{t-\ell} \\
& = (-1)^{k-1} \sum_{\substack{\mu\vdash n+k \\ \mu_i\geq1 }} \binom{n+k}{\mu_1,\dots,\mu_k}   
m_{\mu}(x_1,\dots,x_k)   \\
& = (-1)^{k-1} \frac{(n+k)!}{n!}x_1\cdots x_k \sum_{\lambda\vdash n} \binom{n}{\lambda_1,\cdots,\lambda_k} 
\frac{m_{\lambda} (x_1,\dots,x_k)}{(\lambda_1+1)\cdots(\lambda_k+1)} \\
& = \frac{(-1)^{k-1}(n+k)! \cdot \prod_{j=1}^kx_j}{n!} \left(\sum_{i=1}^k\mathcal{D}_i x_i\right)^n,
\end{align*}
where the last equality is due to \eqref{mono}. The claim follows by applying \eqref{a10}.
\end{proof}

The proof of Theorem~\ref{Conjecture1} will also require the following fact about products of formal power series and the (unique) involution $\omega$ on symmetric functions, that is determined by the action $\omega(p_n)=(-1)^{n-1}p_n$ on power sum functions 
(see, for example \cite[Chapter 7]{RPS}). 

\begin{lemma} \label{invol} If $f(z)$ is a formal power series with constant term $1$ and $g(z):=\frac1{f(-z)}$, then 
$$\omega\left(\prod_if(zx_i)\right)=\prod_i g(zx_i).$$
\end{lemma}

\begin{proof} We begin the proof by letting
$$L(z)=\log f(z)=\sum_{k\geq1} L_kz^k,
$$
which in turns give $\log g(z)=-L(-z).$
Therefore, we find that
\begin{align*} \log\left(\prod_i f(zx_i)\right) & =\sum_i\log f(zx_i) = \sum_i \sum_k L_k x_i^kz^k=\sum_k L_kp_kz^k,
\end{align*}
where $p_k$ is the $k$th power sum symmetric function. Applying $\omega,$ we obtain
\begin{align*} \omega\left(\log\left(\prod_i f(zx_i)\right)\right) & = \sum_k (-1)^{k-1}L_kp_kz^k = -\sum_i L(-zx_i)=\sum_i \log g(zx_i).
\end{align*}
The claimed product formula follows by exponentiation.
\end{proof}

\begin{proof}[Proof of Theorem~\ref{Conjecture1}]
Thanks to Lemma ~\ref{lemma1} (3) and (4), respectively, we have
$$\left(p_1+\sum_i \mathcal{C}_ix_i\right)^r=\left(\sum_i\mathcal{D}_ix_i\right)^r \qquad \text{and} \qquad
\left(p_1+\sum_i \mathcal{B}_ix_i\right)^r=\left(\sum_i(-\mathcal{B}_i)x_i\right)^r.
$$ 
By expressing these expressions in terms of the power sum bases for symmetric functions (utilizing the two umbral representations \eqref{black} and Lemma \ref{conjTA}), we obtain
\begin{align*} 
\left (\sum_i\mathcal{D}_ix_i\right)^r & = T_r(p_1,p_2,\dots) = \sum_{\lambda\vdash r} g^{\lambda} p_{\lambda}(\pmb{x}), \\
\left (\sum_i\widetilde{\mathcal{B}}_ix_i\right )^r & = f_r(p_1,p_2,\dots) = \sum_{\lambda\vdash r} \widetilde{g}^{\lambda} p_{\lambda}(\pmb{x}),
\end{align*}
for some coefficients $g^{\lambda}$ and $\widetilde{g}^{\lambda}$ where $\pmb{x}=(x_1,x_2,\dots)$ and $\widetilde{\mathcal{B}}_i=-\mathcal{B}_i$.

Next, one employs Lemma~\ref{invol} with the exponential generating functions 
$$f(z)=\frac{\exp(z)-1}z \qquad {\text {\rm and}}\qquad g(z)=\frac{z}{1-\exp(-z)},
$$
respectively, for $\mathcal{D}_k=\frac1{k+1}$ and $(-1)^kB_k$. By extracting the degree $r$ terms (in $z$) and noting that $\omega(p_k)=(-1)^{k-1}p_k$, we obtain the desired outcome that
$$T_r(p_1,p_2,p_3,p_4,\dots)=f_r(p_1,-p_2,p_3,-p_4,\dots).$$
\end{proof}

\subsection{Proof of Theorem~\ref{T_and_Y}}

Recalling (\ref{U_Trace}) and the identification $\Psi(E_{2n}(q))=p_{2n}(\pmb{x}^k),$ we define
$$u_{2n}(\pmb{x}^k):=\Psi(U_{2n}(q))
$$
and
\begin{align} \label{u_tilde}
\widetilde{u}_{2n}(\pmb{x}^k)
&:= \sum_{\lambda\vdash n} \phi_U(\lambda)\, \prod_{j=1}^n\Psi(-E_{2j}^{m_j}) \qquad
\text{whenever $\lambda=(1^{m_1},2^{m_2},\dots,n^{m_n})$}.
\end{align}

\begin{example} We list the first few of the above two sequences of polynomials:
\begin{align*} u_0&=1, \,\,\, \,\,  u_2=p_2, \,\,\, \,\, u_4=\frac{5p_2^2-p_4}3, \,\,\, \,\,  u_6=\frac{35p_2^3-42p_2p_4+16p_6}9. \\
\widetilde{u}_0&=1, \,\,\, \,\,  \widetilde{u}_2=- p_2, \,\,\, \,\,  \widetilde{u}_4=\frac{5p_2^2+p_4}3, \,\,\, \,\,  \widetilde{u}_6=-\frac{35p_2^3+42p_2p_4+16p_6}9. \\
\end{align*}
\end{example}

\begin{remark}
In this section, we make important use of P\'olya's cycle index theorem (PCIT) \cite[Section 5]{RPS}
$$\exp\left(\sum_{j\geq1}Z_j\,\frac{w^j}j\right)
=\sum_{n\geq0}\left(\sum_{\lambda\vdash n}  
\prod_{j=1}^n \frac1{m_j!} \left(\frac{Z_j}{j}\right)^{m_j} 
\right)w^n,
$$
where we require writing a partition in the \emph{frequency notation}  $\lambda=(1^{m_1}\dots t^{m_n})\vdash n$. PCIT is deeply intertwined with other key results and methods, offering a consistent approach to a variety of counting problems that are similar and adjacent to the results in this paper. Specifically, it provides a method for deriving generating functions for structures related to generalized Bernoulli and Bell polynomials \cite{Matsusaka}, Bernoulli-Barnes polynomials \cite{BB} (with interesting connection to Fourier Dedekind sums), and it forms a crucial part of the standard exponential generating functions due to Touchard \cite[ Eq'n. (5.30)]{RPS}.
\end{remark}

We proceed with some relevant preparatory results. The first of which, Lemma ~\ref{lm1conj5}), is an explicit and compact umbral representations of $u_{2n}$ and $\widetilde{u}_{2n}$ (hence of $U_{2n}$, by association) in the form of
$$u_{2n} =(2n+1) \left( \sum_{i=1}^k (1+2\mathcal{C}_i)x_i\right)^{2n} \qquad \text{and} \qquad 
\widetilde{u}_{2n} =(2n+1) \left( \sum_{i=1}^k (1+2\mathcal{B}_i)x_i\right)^{2n},$$
where $\mathcal{C}$ and $\mathcal{B}$ are defined in Lemma~\ref{lemma1}.
In this context, observe our use of the umbral mechanism for a $k$-tuple of symbol: 
$$\mathcal{B}_1^{a_1} \mathcal{B}_2^{a_2} \cdots\mathcal{B}_k^{a_k} \longrightarrow
B_{a_1} B_{a_2} \cdots B_{a_k},
$$
where $a_1 + a_2 + \cdots + a_k=n$.

\begin{lemma} \label{lm1conj5} For each nonnegative integer $n$, we have
\begin{align*}
\left(\sum_{i=1}^k (1+2\mathcal{C}_i)x_i\right)^n & = \begin{cases} \frac{u_n}{n+1} \qquad & \text{if $n$ is even} \\  \,\,\, 0 \qquad & \text{if $n$ is odd}, \end{cases}  \\
\left(\sum_{i=1}^k (1+2\mathcal{B}_i)x_i\right)^n & = \begin{cases} \frac{\widetilde{u}_n}{n+1} \qquad & \text{if $n$ is even} \\  \,\,\, 0 \qquad & \text{if $n$ is odd}. \end{cases}
\end{align*}
\end{lemma}

\begin{proof} We only prove the second assertion, as the first one follows a similarly. To this end, we recall the series expansion \cite[1.518.1]{GR} for the reciprocal of what we call here the \emph{$\text{sinhc}$ function} 
$\text{sinhc}(t)=\frac{\sinh(t)}t$ (reminiscent of the \emph{sinc function}) given by
\begin{align} \label{sinc} \frac1{\text{sinhc}(t)} 
= \exp\left(-\sum_{n\geq1} \frac{4^n B_{2n}}{(2n)(2n)!}\cdot t^{2n}\right). \end{align}
On the other hand, the (usual) exponential generating function for the Bernoulli numbers takes the form
$$\exp(\mathcal{B}z)=\sum_{n\geq0} \mathcal{B}^n\frac{z^n}{n!} = \sum_{n\geq0} B_n \frac{z^n}{n!} = \frac{z}{\exp(z)-1} $$
from which it is natural to deduce that
\begin{align} \label{bernou} \exp((1+2\mathcal{B}_i)x_iz) = \sum_{n\geq0} ((1+2\mathcal{B}_i)x_i)^n\frac{z^n}{n!} = \frac{2x_iz\cdot \exp(x_iz)}{\exp(2x_iz)-1} = \frac1{\text{sinhc}(x_iz)}. \end{align}
Consequently, \eqref{bernou} implies the finite product 
\begin{align} \label{sincprod}
\prod_{i=1}^k\frac1{\text{sinhc}(x_iz)}  & = \exp\left(z\sum_{i=1}^k (1+2\mathcal{B}_i)x_i\right) 
 = \sum_{n\geq0} \frac{z^n}{n!} \left(\sum_{i=1}^k(1+2\mathcal{B}_i)x_i \right)^n.
\end{align}
Combining \eqref{sinc}, \eqref{sincprod} and P\'olya's cycle index theorem \cite[Section 5]{RPS}, we arrive at
\begin{align*}
\sum_{n\geq0} \frac{z^n}{n!} \left(\sum_{i=1}^k (1+2\mathcal{B}_i)x_i \right)^n 
&= \prod_{i=1}^k\frac1{\text{sinhc}(x_iz)}  
 =  \exp\left(-\sum_{n\geq1} \frac{4^n B_{2n} \, p_{2n}}{(2n)(2n)!}\cdot z^{2n}\right) \\
& = \sum_{n\geq0} z^{2n} 
\sum_{\lambda\vdash n} \left(\prod_{j=1}^n \frac1{m_j!}\left(\frac{- 4^j B_{2j}}{(2j)(2j)!}\right)^{m_j}\right) \Psi(E_{\lambda}).
\end{align*}
Compare the coefficients of $z^n$ on both sides, together with \eqref{u_tilde}, to complete the proof.
\end{proof}

In \cite[(25)]{Fel}), one of the authors proved the identity
\begin{align} \label{Fel_f} f_{2n+1} =  \sum_{j=1}^{2n+1} (-1)^{j+1} \binom{2n+1}j  f_1^j f_{2n+1-j}. \end{align}
The next result reproves \eqref{Fel_f} and extends it to all $f_n$.

\begin{lemma} \label{lm2conj5} For each positive integer $n$, we have the recursive formula that
$$f_n =  \begin{cases}  \sum_{j=1}^n (-1)^{j+1} \binom{n}j  f_1^j f_{n-j} \qquad & \text{if $n$ is odd}, \\
\sum_{j=1}^n (-1)^{j+1} \binom{n}j  f_1^j f_{n-j} + \frac{\widetilde{u}_n}{(n+1)2^n} \qquad & \text{if $n$ is even}. \end{cases}
$$
\end{lemma}

\begin{proof} Letting $\mathfrak{f}:=p_1+\sum_{i=1}^k\mathcal{B}_ix_i$ The use of umbral symbolism allows for the more compact reformulation
\begin{align} \label{f_umbral} (\mathfrak{f}-f_1)^n =  \begin{cases} \,\,\,\,\,\, 0 \qquad & \text{if $n$ is odd}, \\  \frac{\widetilde{u}_n}{(n+1)2^n} \qquad & \text{if $n$ is even}. \end{cases}
\end{align}
It suffices to only consider the case $n$ is even. Recalling  \eqref{black}, expand the left-hand side as
\begin{align*}
(\mathfrak{f}-f_1)^{2n} & = \left(p_1+\sum_{i=1}^k\mathcal{B}_ix_i -\frac12p_1\right)^{2n} = \frac1{4^n}\left(\sum_{i=1}^k (1+2\mathcal{B}_i)x_i \right)^{2n}
= \frac{\widetilde{u}_{2n}}{(2n+1)4^n},
\end{align*}
where the last equality is due to Lemma ~\ref{lm1conj5}.
\end{proof}

We have now assembled enough results to help us prove the next instrumental lemma.

\begin{lemma}  \label{conj5} For each integer $n\geq0$, we have the representations 
$$T_n=\frac1{(n+1)2^n} \sum_{j=0}^{\lfloor\frac{n}2\rfloor} \binom{n+1}{2j+1} p_1^{n-2j} 
u_{2j} \qquad \text{{\rm and}} \qquad
f_n=\frac1{(n+1)2^n} \sum_{j=0}^{\lfloor\frac{n}2\rfloor} \binom{n+1}{2j+1} p_1^{n-2j}\, \widetilde{u}_{2j}.$$
\end{lemma}

\begin{proof} We proceed by induction on $n$. The base cases $n=0$ and $n=1$ are routine and hence omitted.  If we revert to the umbral mechanics, in the variable $\widetilde{u}$ only, the assertion amounts to $f_n=\frac1{(n+1)2^n} \left\{\frac{(p_1+\widetilde{u})^{n+1}-(p_1-\widetilde{u})^{n+1}}{2\widetilde{u}}\right\}$. 

\smallskip

\bf
\noindent The case $n$ odd: \rm We know from \eqref{Fel_f} and Lemma \ref{lm2conj5} that
$$f_{2n+1}= \sum_{j=1}^{2n+1} (-1)^{j+1} \binom{2n+1}j  f_1^j f_{2n+1-j}.$$
 In light of this, we may continue using the induction hypothesis and the seed $f_1=\frac12p_1$ to obtain
\begin{align*}
f_{2n+1} 
& = \sum_{j=1}^{2n+1} \frac{ (-1)^{j+1} \binom{2n+1}j p_1^j}{(2n+2-j)2^{2n+1}} \left\{\frac{(p_1+\widetilde{u})^{2n+2-j}-(p_1-\widetilde{u})^{2n+2-j}}{2\widetilde{u}}\right\}   \\
& = \frac1{(2n+2)2^{2n+1}} \sum_{j=1}^{2n+1} (-1)^{j+1} \binom{2n+2}j p_1^j \left\{\frac{(p_1+\widetilde{u})^{2n+2-j}-(p_1-\widetilde{u})^{2n+2-j}}{2\widetilde{u}}\right\}  \\
& = \frac1{(2n+2)2^{2n+1}} \left\{\frac{(p_1+\widetilde{u})^{2n+2}-(p_1-\widetilde{u})^{2n+2}}{2\widetilde{u}}\right\}.
\end{align*}
\bf The case $n$ even: \rm The argument is similar. We apply Lemma ~\ref{lm2conj5} and induction to write
\begin{align*}
f_{2n} 
& =  \frac{\widetilde{u}_{2n}}{(2n+1)4^n}+ \sum_{j=1}^{2n} \frac{(-1)^{j+1}} {(2n+1)2^{2n}} \binom{2n+1}j p_1^j \left\{\frac{(p_1+\widetilde{u})^{2n+1-j}-(p_1-\widetilde{u})^{2n+1-j}}{2\widetilde{u}}\right\}  \\
& = \frac1{(2n+1)2^{2n}} \left\{\frac{(p_1+\widetilde{u})^{2n+1}-(p_1-\widetilde{u})^{2n+1}}{2\widetilde{u}}\right\}.
\end{align*}
This completes the proof.
\end{proof}

\begin{proof}[Proof of Theorem~\ref{T_and_Y}] The proof follows by combining two structural results.
The first one concerning $T_n(\pmb{x}^k)$ comes from Lemma~\ref{conj5}. The other is our principal result from Theorem~\ref{Omega_coeff2} in regard to $Y_n(q)$. Their common formulation is based on \eqref{Yn_tilde}, which is due to the mapping $\Psi$.
\end{proof}

\subsection{Proof of Theorem~\ref{Conjecture2}}
In this subsection, we begin by recalling some calculations found in \cite[p. 59, (13)]{Fel}, namely that
\begin{align*}
\frac{T_3(\pmb{x}^k)}{T_1^3(\pmb{x}^k)} & = 3 \frac{T_2(\pmb{x}^k)}{T_1^2(\pmb{x}^k)}  -2 , \\
\frac{T_5(\pmb{x}^k)}{T_1^5(\pmb{x}^k)} & = 5 \frac{T_4(\pmb{x}^k)}{T_1^4(\pmb{x}^k)}  -20 \frac{T_2(\pmb{x}^k)}{T_1^2(\pmb{x}^k)}  +16, \\
\frac{T_7(\pmb{x}^k)}{T_1^7(\pmb{x}^k)} & = 7 \frac{T_6(\pmb{x}^k)}{T_1^6(\pmb{x}^k)} - 70 \frac{T_4(\pmb{x}^k)}{T_1^4(\pmb{x}^k)} 
+336 \frac{T_2(\pmb{x}^k)}{T_1^2(\pmb{x}^k)}  -272.
\end{align*}

The author \cite{Fel} also derived a similar formulation for the polynomials 
$f_n(\pmb{x}^k)$. However, a compact formula is missing from \cite[(25)-(26)]{Fel} and we fill that gap. 
For integers $n\geq1$, the modified version of \cite[page 61,  (25)]{Fel} is
\begin{align}  \label{Fel25}
\frac{f_{2n+1}(\pmb{x}^k)}{f_1^{2n+1}(\pmb{x}^k)}
= \sum_{j=0}^n (-1)^j A_{2j+1}\binom{2n+1}{2j+1} \frac{f_{2n-2j}(\pmb{x}^k)}{f_1^{2n-2j}(\pmb{x}^k)}.
\end{align}
Theorem~\ref{Conjecture2} is then an immediate consequence of Theorem~\ref{Conjecture1} and \eqref{Fel25} above.


\begin{thebibliography}{99}


\bibitem{AOS} T. ~Amdeberhan, K. Ono, A. Singh,
\emph{Derivatives of theta functions as traces of partition Eisenstein series}, 
Nagoya Mathematical Journal  (doi:10.1017/nmj.2024.30). 2025:1-12.

\bibitem{AGOS} T. ~Amdeberhan, M. ~Griffin, K. ~Ono, A. ~Singh, 
\emph{Traces of partition Eisenstein series}, 
Forum Mathematicum, {\bf 37} no. 6, 2025, pp. 1835-1847. 

\bibitem{AGO} T. ~Amdeberhan, M. ~Griffin, K. ~Ono, 
\emph{Some topological genera and Jacobi's theta function}, 
Proc. National Academy of Sciences,  USA, V. 122, No. 32 (2025).

\bibitem{Andrews} G. E. Andrews, \emph{Theory of partitions}, Cambridge Univ. Press, 1998.

\bibitem{AndrewsBerndt} G. E. Andrews and B. C. Berndt, \emph{Ramanujan's Lost Notebook, Part II}, Springer, New York, 2009.

\bibitem{BB} A. ~Bayad, M. Beck,
\emph{Relations for Bernoulli-Barnes numbers and Barnes zeta functions},
Int. J. Number Theory {\bf 10} (2014), no. 5, 1321-1335.

\bibitem{Berndt1} B. C. Berndt, S. H. Chan, Z.-G. Liu, and H. Yesilyurt, \emph{A new identity for $(q;q)_{\infty}^{10}$ with an application to Ramanujan's partition congruence modulo 11},  Quart. J. Math. \textbf{55} (2004), 13-30.


\bibitem{Berndt2} B. C. Berndt and A. J. Yee, \emph{A page on Eisenstein series in Ramanujan's lost notebook}, Glasgow Math. J. \textbf{45} (2003), 123-129.

\bibitem{BPvI} K. Bringmann, B. Pandey, and J.-W. van Ittersum, \emph{Mock Eisenstein series associated to partition ranks}, preprint.


\bibitem{fl17} L. G. Fel, 
\emph{Restricted partition functions and identities for degrees of syzygies in numerical semigroups}, 
Ramanujan J., {\bf 43} (2017), 465-491.

\bibitem{fl22} L. G. Fel,
\emph{Genera of numerical semigroups  and polynomial identities for degrees of syzygies}, 
In Combinatorial and Additive Number Theory 5, 
Springer Proceedings in Mathematics \& Statistics, (2022), 153-178.

\bibitem{Fel} L. G. ~Fel,
\emph{Symmetric poynomials associated with numerical semigroups},
Discrete Math. Lett. {\bf 5} (2021), 56-62.

\bibitem{fr86} R. ~Fr\"oberg, C. Gottlieb and R. ~H\"aggkvist, 
\emph{On numerical semigroups}, Semigroup Forum, {\bf 35}, 63-83, (1986).

\bibitem{GR} I. S.~Gradshteyn and I. M.~Ryzhik, 
\emph{Table of integrals, series, and products. Translated from the Russian. Translation edited and with a preface by Victor Moll and Daniel Zwillinger}. 
Amsterdam: Elsevier/Academic Press, 8th updated and revised ed. edition, 2015.

\bibitem{Hitchin} N.~Hitchin, 
\emph{The Dirac operator}, Bull. Amer. Math. Soc. 67 no. 1 (2025), 3-16.

\bibitem{Zagier} M. ~Kaneko and D. ~Zagier,
\emph{A generalized Jacobi theta function and quasimodular forms},
The moduli space of curves (Texas Island, 1994), Progr. Math., vol. \textbf{129}, Birkh\"auser Boston, Boston, MA (1995), 165--172.

\bibitem{Knuth} D. Knuth, \emph{Convolution polynomials}, Mathematica {\bf 2} (1992), no. 4, 67--78.

\bibitem{Matsusaka} T. Matsusaka, {\emph Applications of Faà di Bruno’s formula to partition traces},  Research in Number Theory {\textbf 11}, 69 (2025). https://doi.org/10.1007/s40993-025-00651-9.


\bibitem{CBMS}  K. Ono,
 \emph{The web of modularity: arithmetic of the coefficients of modular forms and $q$-series}, CBMS Regional Conference Series in Mathematics, \textbf{102}, Amer. Math. Soc., Providence, RI, 2004.

\bibitem{Rama} S. ~Ramanujan, 
\emph{The lost notebook and other unpublished papers}, (1988)
New Delhi; Berlin, New York: Narosa Publishing House; Springer-Verlag, Reprinted (2008).

\bibitem{RR} S. M. ~Roman, G-C.~Rota,
\emph{The umbral calculus}, Adv. Math. {\bf 27} (2) (1978), 95-188.

\bibitem{RPS} R. P. ~Stanley, 
\emph{Enumerative combinatorics. Vol. 2}, Cambridge Studies in Math. {\bf 62} Cambridge Univ. Press, 1999.

\bibitem{st78}  R.P. ~Stanley, 
\emph{Invariants of finite groups and their applications to combinatorics}, 
Bull. Amer. Math. Soc. (N.S.) {\bf 1} (1979), 475-511.

\end{thebibliography}
\end{document}